\documentclass[11pt]{article}
\usepackage{fullpage}
\usepackage[psamsfonts]{eucal}
\usepackage{amsmath}
\usepackage{amsfonts}
\usepackage{amssymb}
\usepackage{amstext}
\usepackage[symbol]{footmisc}
\usepackage[noadjust]{cite}
\usepackage{amscd}
\usepackage{amsthm}
\usepackage{makeidx}
\usepackage{tikz}
\usepackage{graphicx}
\usepackage[colorlinks=true,linkcolor=blue]{hyperref}
\usetikzlibrary{calc}

\usepackage{supertabular}
\usepackage{enumitem}
\usepackage{titling}
\usepackage{wasysym}

\usepackage{comment}
\usepackage{cleveref}
\usepackage{listings}
\usepackage{relsize}

\newcommand{\RR}{\mathbb{R}}

\newcommand{\PP}{\mathbb{P}}

\newcommand{\Tr}{\mathrm{Tr}}

\usepackage{microtype}

\theoremstyle{plain}
\newtheorem{thm}{Theorem}[section]

\newtheorem{prop}[thm]{Proposition}
\newtheorem{lemma}[thm]{Lemma}

\newtheorem{dfn}[thm]{Definition}

\newtheorem{rmk}[thm]{Remark}
\newtheorem{cor}[thm]{Corollary}

\newcommand{\Var}{\textnormal{Var}}

\def\title#1{
    \thispagestyle{plain}
    \vspace*{-14pt}
    \vskip 39pt
    {\centering{\titlefont #1\par}}
    \vspace*{4pt}
    \vskip 1em
}
\def\author#1{\par
    {\centering{\authorfont#1}\par\vspace*{0.05in}}
}
\def\address#1{\par
    {\centering{\affiliationfont#1\par}}\par\vspace*{11pt}
}

\def\titlefont{\fontsize{14}{21}\boldmath\selectfont\centering{}}
\def\authorfont{\fontsize{13}{15}}

\let\affiliationfont\rhfont

\definecolor{forestgreen}{RGB}{34, 139, 34}

\newcommand{\bR}{\mathbb{R}}
\newcommand{\bC}{\mathbb{C}}

\newcommand{\deq}{\mathrel{\mathop:}=}

\newcommand{\bP}{\mathbb{P}}

\newcommand{\bE}{\mathbb{E}}

\newcommand{\be}{\begin{equation}}
\newcommand{\ee}{\end{equation}}

\numberwithin{equation}{section}

\newcommand{\rd}{{\rm d}}

\hypersetup{pdftitle={A lower bound for non-backtracking eigenvalues from a Galton-Watson local limit}}
\allowdisplaybreaks[1]

\def\titlefont{\fontsize{15}{17}\bfseries\boldmath\selectfont\centering{}}
\def\authorfont{\fontsize{13}{15}}

\let\affiliationfont\rhfont

\def\address#1{\par
    {\centering{\affiliationfont#1\par}}\par\vspace*{11pt}
}

\def\body{
\setcounter{footnote}{0}
\def\thefootnote{\alph{footnote}}
\def\@makefnmark{{$^{\rm \@thefnmark}$}}
}

\def\title#1{
    \thispagestyle{plain}
    \vspace*{-14pt}
    \vskip 79pt
    {\centering{\titlefont #1\par}}
    \vskip 1em
}

\begin{document}

\title{An Alon-Boppana Bound for the Non-Backtracking Operator}
\begin{center}

\begin{minipage}[c]{0.2\textwidth}
 \author{Theo McKenzie}
\address{Yale University\\
   theo.mckenzie@yale.edu }
 \end{minipage}
 \end{center}

\begin{abstract}
For any fixed \(k\), we prove a lower bound on the \(k\)th largest modulus of an eigenvalue of the non-backtracking matrix \(B\). Specifically, consider any deterministic or random family of graphs that converges locally to the unimodular Galton–Watson tree with root degree distribution \(D\), and set \(\kappa:=\mathbb E[D(D-1)]/\mathbb E[D]\). Given $\kappa>1$ and an exponential-moment bound on the empirical degree distributions, we show that \(|\lambda_k(B)|\geq\sqrt{\kappa}-o_N(1)\), where \(N\) is the number of vertices. When restricted to locally tree-like regular graphs, this recovers a well-known consequence of the Ihara-Bass formula. In the specific case where the graph is generated through the Erdős–Rényi model with expected degree \(d>1\), this proves a conjecture of Bordenave, Lelarge, and Massoulié.

 To do this, we show that the normalized log-determinant of the Bethe-Hessian of the graph is bounded by that of the Bethe-Hessian of its local limit. This bound is violated if the eigenvalues of the non-backtracking matrix are too small. We establish this using an effective-conductance
interpretation of the tree Green's function recursion.

\end{abstract}
\section{Introduction}
The non-backtracking matrix is a useful tool in community detection
\cite{krzakala2013spectral,saade2014spectral,bordenave2018nonbacktracking},
random matrix theory \cite{sodin2007random}, the analysis of mixing
times \cite{lubetzky2016cutoff,alon2007non}, and the study of adjacency
spectra \cite{angel2015non,bordenave2020new}. It is indexed by the
directed edges of a graph and records which edges can be traversed
consecutively, subject to the restriction that a walk cannot
immediately retrace an edge. The definition is as follows. Let $G$ be a (possibly randomly sampled) graph on $N$ vertices. Its non-backtracking matrix $B_G$ is indexed by oriented
edges and has entries
\be\label{eq:nbw}
 (B_G)_{(u,v),(x,y)}
 =\mathbf 1_{\{v=x\}}\mathbf 1_{\{y\ne u\}}.
\ee

 Perhaps surprisingly, the spectral behavior of this directed graph can be quite different from that of the original adjacency operator.  One way to analyze the spectrum of the non-backtracking operator is through the Ihara-Bass identity \cite{hashimoto1989zeta,bass1992ihara}. Write $A_G$ for the adjacency matrix of $G$, and $D_G$ for its diagonal matrix of degrees. For $r\in \bC$, we write write the \emph{Bethe-Hessian}
\begin{equation}\label{eq:BH}
 H(r,G)=D_G-I-rA_G+r^2I.
\end{equation}
Define $N$ to be the number of vertices of $G$, and $|E_N|$ to be the number of edges. The Ihara-Bass identity  is
\begin{equation}\label{eq:ihara}
 \det(rI-B_G)=(r^2-1)^{|E_N|-N}\det H(r,G).
\end{equation}

In the case where $G$ is a regular graph, $D_G$ is a multiple of the identity. Therefore, \eqref{eq:ihara} gives that, in this regular case, besides the trivial $|E_N|-N$ eigenvalues at $\pm1$, eigenvalues $\lambda$ of the non-backtracking matrix are of the form 
\be \label{eq:reggraph}
\lambda=\frac{\mu\pm \sqrt{\mu^2-4(d-1)}}{2},
\ee
for some eigenvalue $\mu$ of $A_G$. 

For non-regular graphs, $D_G$ and $A_G$ need not commute, so the
Ihara-Bass identity does not yield a formula such as
\eqref{eq:reggraph} in terms of the adjacency eigenvalues alone.
Moreover, since $B_G$ is non-Hermitian, the usual arguments relating
local weak convergence to convergence of Hermitian spectral measures
do not apply directly.

A natural first question about the non-backtracking spectrum is whether it satisfies
an analogue of the \emph{Alon-Boppana bound} for adjacency matrices. For fixed $d$ and $k\ge1$, this bound
states that the $k$th largest adjacency eigenvalue of an $N$-vertex
$d$-regular graph is at least $2\sqrt{d-1}-o_N(1)$ \cite{alon1986eigenvalues}.
The bound has received numerous strengthenings, as well as extensions to non-regular graphs
\cite{nilli1991second,hoory2005lower,jiang2019spectral}.
For the non-backtracking matrix of a \emph{regular} graph, a corresponding
lower bound follows directly from \eqref{eq:reggraph}: the two roots
associated with each adjacency eigenvalue have product $d-1$,
so at least one has modulus at least $\sqrt{d-1}$.

An analogous lower bound for non-regular graphs was conjectured by
Bordenave, Lelarge, and Massouli\'e
\cite{bordenave2018nonbacktracking}. For an Erd\H{o}s-R\'enyi graph
$G\sim G(N,\alpha/N)$ with fixed $\alpha>1$, they proved that
$|\lambda_2(B_G)|\le\sqrt{\alpha}+o_N(1)$ with high probability
and predicted a matching lower bound.
This question arose from their analysis of community detection
in stochastic block models with common asymptotic expected degree
$\alpha$. Under their hypotheses, the Perron eigenvalue and
informative outliers are separated from the remaining spectrum,
where eigenvalues have modulus at most $\sqrt{\alpha}+o_N(1)$. They also give a matching lower bound on the $k$th largest singular value. 
Such a lower bound on the eigenvalue would identify $\sqrt{\alpha}$ as the
actual spectral edge underlying this separation. Further evidence of a relationship was given when Anantharaman obtained comparisons between the mixing rates of the simple random walk and the non-backtracking walk, assuming degrees are bounded above and are at least 3 \cite{anantharaman2017some}. 

In this work, we confirm this conjecture by giving a result that generalizes the conjecture and the known bound resulting from \eqref{eq:reggraph} into a single result. Specifically, we show for a family of graphs that converges to a Galton-Watson limit and any fixed $k$, the $k$th largest modulus of an eigenvalue is bounded from below by the square root of the rate of expansion of the tree (see \Cref{thm:main}).

For the adjacency operator, this type of bound is proven by considering a family of test functions around the local neighborhood of a vertex, and showing these converge to the spectral radius of the local limit. For example, this leads to the bound $2\sqrt{d-1}$ in the regular case, as this is the spectral radius of the infinite
$d$-regular tree. Such a method is not effective for the non-backtracking matrix. For example, if we pass to local neighborhoods that are trees, the non-backtracking matrix is nilpotent.

The goal instead is to use the Bethe-Hessian to control the spectrum. In order to get full spectral information about our graph, we would typically need to obtain information about $H(r,G)$ for complex $r$, which complicates the problem as it makes $H(r,G)$ non-Hermitian. However, we show that by passing to the log-determinant for a single, well-chosen $r$, the contribution of eigenvalues with small absolute value is explicit. We then compare this to the log-determinant of the tree limit, which we can analyze sharply. If the eigenvalues of $B_G$ are too small, then we violate the convergence of the log-determinant guaranteed by weak convergence.

 Given this goal, the majority of the remaining analysis is to prove the law of the log-determinant of the local tree limit. To do this, we use the classical correspondence between electrical
networks and Green functions: for a network Laplacian with Dirichlet
boundary conditions, the diagonal Green function equals the effective
resistance to the grounded boundary, and its reciprocal is the effective
conductance 
\cite{Keller2017criticality}. This correspondence has also been
used to obtain spectral information. In particular, in \cite{Lyons2010identities} the identification of the diagonal resolvent with effective resistance gives an expression of the log-determinant.

We connect these observations to the present problem by identifying
our zero-energy Bethe-Hessian cavity variable with a normalized
effective conductance. The cavity recursion then becomes the
series-parallel reduction rule for the weighted tree.
This identification allows us to establish positivity of the cavity
variable by constructing a unit flow of finite energy. To complete this analysis, we separate the log-determinant into vertex and edge components, similar to analysis done to find the spectrum of Jacobi matrices on trees through the log-determinant \cite{banks2024useful}.
\paragraph{Acknowledgement}
The question of finding a lower bound on the second largest eigenvalue of the nonbacktracking spectrum in the Erd\H{o}s-R\'{e}nyi model was first presented to the author by Charles Bordenave at the Centre International de Recontres Math\'{e}matiques (CIRM) in Marseille, France,  at the program ``Random Hyperbolic Surfaces and Random Graphs''.  

The author used OpenAI's GPT-6 in developing the proof. The author asked the model for ideas to prove the main theorem for the Erd\H{o}s-R\'{e}nyi graph using the Bethe-Hessian, in response to which it came up with the fundamental idea of using the log-determinant to pass to a Hermitian matrix and make the contribution of small eigenvalues fixed. After further prompting, it created much of the analysis used to control the log-determinant, including interpreting the resolvent as a conductance and estimating the escape probability. It was also used to draft and edit certain arguments.

\section{Statement of the result}

Let $D$ be a nonnegative integer-valued random variable with
$0< m :=\bE D<\infty$. $D$ defines the distribution of the degree of a randomly selected vertex, but most of our analysis deals with the expansion at later levels from this root. For example, for a regular tree, the root has degree $d$, but the size of a ball scales at rate $(d-1)$.  Therefore, define the forward offspring distribution (see \cite[Section 2.2]{bordenave2010resolvent}) as
\begin{equation}\label{eq:sizebiased-law}
 \PP(K=k)=\frac{(k+1)\PP(D=k+1)}{ m },\qquad k\ge0,
 \qquad
 \kappa:=\bE K=\frac{\bE[D(D-1)]}{ m }.
\end{equation}
Let $(T,o)$ be the random rooted tree whose root has $D$ children and whose
other vertices independently have $K$ children. The root degree and all
these offspring variables are independent. This is the unimodular
Galton-Watson tree associated with $D$.

Next, we give our condition on the graph family.
\begin{dfn}[Graph convergence condition]

For each \(L\geq 1\), call \(f_L(G,v)\in\RR\) an $L$-local function if it is a function of the distance $L$ neighborhood of $v$ in the rooted graph $(G,v)$. By \emph{empirical
local convergence in probability} of the family of graphs $\{G_N\}$ to $(T,o)$, we mean the following. For every fixed $L\geq 0$ and $L$-local function \(f_L(G,v)\), such that $\frac1N\sum_{v\in V(G_N)} |f_L(G_N,v)|^2=O_{\bP}(1)$, we have

\begin{equation}\label{eq:local}
 \frac1{N}\sum_{v\in V(G_N)}f_L(G_N,v)
 \xrightarrow{\PP}\bE f_L(T,o).
\end{equation}

\end{dfn}
Note this is a condition on the empirical neighborhood distribution, not merely
on the neighborhood of a random vertex after averaging over $G_N$. 

Throughout, we define $B_N$ to be the nonbacktracking matrix of $G_N$.
We order the eigenvalues of the non-backtracking matrix by nonincreasing
modulus, counting algebraic multiplicity. We then have 
\[
 \lambda_1(B_N)\ge|\lambda_2(B_N)|\ge|\lambda_3(B_N)|\ge\cdots.
\]
When the matrix has fewer than $k$ eigenvalues, set $\lambda_k(B_N)=0$.

The following is the main result of this work, which gives an Alon-Boppana type bound based on the local limit. 
\begin{thm}\label{thm:main}
For each $N$, let $G_N$ be a random graph on $N$ vertices, and suppose
that the sequence $\{G_N\}$ satisfies \eqref{eq:local} with $\kappa>1$. Assume that, for some
$a>0$,
\begin{equation}\label{eq:exponential}
 \frac1N\sum_{v\in V(G_N)}e^{a\deg(v)}=O_{\PP}(1).
\end{equation}
Then, for any fixed index $k\geq 1$ and error term $\varepsilon>0$,
\begin{equation}\label{eq:main}
 \PP\bigl(|\lambda_k(B_{N})|\le\sqrt{\kappa}-\varepsilon\bigr)
 \longrightarrow 0.
\end{equation}
In particular, if $G_N$ is deterministic, satisfies empirical local
convergence as in \eqref{eq:local}, and obeys
\[
 \sup_N\frac1N\sum_{v\in V(G_N)}e^{a\deg(v)}<\infty
\]
for some $a>0$, then
\begin{equation}\label{eq:deterministic}
 \liminf_{N\to\infty}|\lambda_k(B_{N})|\ge\sqrt{\kappa}.
\end{equation}
\end{thm}

The theorem applies in particular to sparse stochastic block models
with equal asymptotic expected degrees, answering a question of \cite{bordenave2018nonbacktracking} about Erd\H{o}s-R\'{e}nyi and, more generally, the stochastic block model. 

\begin{cor}\label{cor:sbm}
Partition $N$ vertices into a fixed number $q$ of blocks whose
proportions converge to $\pi_1,\ldots,\pi_q>0$. Let $W$ be a fixed
symmetric matrix with nonnegative entries, and form $G_N$ by joining
each pair of distinct vertices in blocks $i$ and $j$ independently
with probability $W_{ij}/N$. Suppose that
\[
 \sum_{j=1}^q W_{ij}\pi_j=\alpha>1
 \qquad \forall i\in\{1,\ldots,q\}.
\]
For any fixed $k\geq 1$, with high probability,
\be\label{eq:sbm}
 |\lambda_k(B_{N})|\geq\sqrt{\alpha}-o_N(1).
\ee

\end{cor}

\begin{proof}[Proof of \Cref{cor:sbm}]
The empirical local limit of $G_N$ is a multitype Galton-Watson
tree in which a vertex of type $i$ has independent
$\operatorname{Pois}(W_{ij}\pi_j)$ numbers of children of type $j$;
see \cite[Proposition 36]{bordenave2018nonbacktracking}.
Since $\sum_j W_{ij}\pi_j=\alpha$ for every $i$, forgetting the types
gives a Galton-Watson tree with $\operatorname{Pois}(\alpha)$
offspring. Thus \eqref{eq:local} holds with $\kappa=\alpha$.

To verify \eqref{eq:exponential}, independence of the edge indicators
gives, for every fixed $a>0$ and every vertex $v$,
\[
 \bE e^{a\deg(v)}
 \le
 \left(1+\frac{e^a-1}{N}\max_{i,j}W_{ij}\right)^{N-1}
 \le
 \exp\left((e^a-1)\max_{i,j}W_{ij}\right).
\]
Averaging over $v$ and applying Markov's inequality yields
\[
 \frac1N\sum_{v\in V(G_N)}e^{a\deg(v)}=O_{\PP}(1).
\]
The conclusion follows from \Cref{thm:main}.
\end{proof}
\begin{rmk}
Assume additionally that the mean offspring matrix
$M=(W_{ij}\pi_j)_{i,j=1}^{q}$ is positively regular, meaning that
some power of $M$ has strictly positive entries, and that the block
proportions converge at rate $O(N^{-\gamma})$ for some $\gamma>0$.
Let $r_0$ be the number of eigenvalues of $M$ with modulus strictly
greater than $\sqrt{\alpha}$, counting algebraic multiplicity.
Combined with \cite[Theorem~4]{bordenave2018nonbacktracking},
\eqref{eq:sbm} then gives
\[
 |\lambda_k(B_N)|\xrightarrow{\PP}\sqrt{\alpha}
 \qquad\text{for every fixed }k>r_0.
\]
\end{rmk}
\begin{rmk}
    Unlike in the regular case, the limiting degree distribution by itself does not imply a nontrivial bound. To see this let $G_N$ be a disjoint union of copies of the star $K_{1,s}$, with
$s>3$ fixed. Its limiting degree law is
\[
 \PP(D=1)=\frac{s}{s+1},\qquad
 \PP(D=s)=\frac1{s+1},
 \qquad
 \frac{\bE[D(D-1)]}{\bE D}=\frac{s-1}{2}>1.
\]
Nevertheless every component is a finite tree, so $B_N$ is nilpotent. Its local limit is a rooted finite star,
not the unimodular Galton-Watson tree specified in the theorem.
\end{rmk}

\section{Consequences of local convergence and the degree bound}

\subsection{A separated evaluation point}

Our first goal will be to show that if the $k$th largest modulus of an eigenvalue is very small, then the log determinant must converge to a specific limit.
We start with some standard bounds on sparse random graphs. 
\begin{lemma}\label{lem:degrees}
With $\Delta_N$ the maximum degree,
\begin{equation}\label{eq:degree-consequences}
\Delta_N=O_{\PP}(\log N),
 \qquad \frac{\rho(B_N)^p}{N}\xrightarrow{\PP}0
 \quad\text{for each fixed }p.
\end{equation}
\end{lemma}
\begin{proof}

By \eqref{eq:exponential}, we
must have $\frac1Ne^{a\Delta_N}=O_{\PP}(1)$, so this gives the first result. Every row sum of $B_N$ is at most $\Delta_N$, so 
$\rho(B_N)\le\Delta_N$. This proves the last statement.
\end{proof}

We next show that for $r$ sufficiently far from the spectrum, the log-determinant must converge to a fixed limit. For simplicity of notation, we now write $H_N(r)\deq H(r,G_N)$.

\begin{lemma}\label{lem:forced}
Fix $k\geq 1$, $1<R<r$ and $\delta>0$. Say $G_N$ is an $N$-vertex graph sampled from our distribution. Define $\sigma(B_N)$ to be the spectrum of $B_N$, and set the event $\Omega_N$ and function $F_0(r)$ as 
\[
 \Omega_N(r,\delta)\deq 
 \{|\lambda_k(B_N)|\le R,\ \min_{\lambda\in \sigma(B_N)}|r-\lambda|\ge\delta\},
 \qquad
 F_0(r)= m \log r-\left(\frac m 2-1\right)\log(r^2-1).
\]
For every $\varepsilon>0$, as  $N\rightarrow \infty$,
\begin{equation}\label{eq:localized}
 \PP\left(\Omega_N(r,\delta)\cap
 \left\{\left|\frac1N\log|\det H_N(r)|-F_0(r)\right|
      >\varepsilon\right\}\right)\longrightarrow0.
\end{equation}
\end{lemma}

\begin{proof}
By \eqref{eq:local}, $|E_N|/N\to m /2>0$ in probability, so we
may discard the event that $B_N$ has fewer than $k$ eigenvalues.
 On $\Omega_N(r,\delta)$, $\det(rI-B_N)$ and, by \eqref{eq:ihara}, $\det H_N(r)$ are nonzero. Considering $\log |1-x|=\operatorname{Re}\log (1-x)$, the power series for logarithms gives the exact identity
\begin{align}\label{eq:trace-log}
 \frac1N\log|\det(rI-B_N)|
 ={}&\frac1N\sum_{\ell=1}^{k-1}\log|r-\lambda_\ell|+\frac{2|E_N|-k+1}{N}\log r\\
 &-\frac1{N}\operatorname{Re}\sum_{p\ge1}
 \frac{\Tr B_N^p-\sum_{\ell=1}^{k-1}\lambda_{\ell}^p}{pr^p}.
\end{align}

For fixed $p$, \Cref{lem:degrees} and \eqref{eq:local} imply that as $N$ increases,
\[
 \frac1{N}\frac{\Tr B_N^p-\sum_{\ell=1}^{k-1}\lambda_{\ell}^p}{pr^p}\xrightarrow{\PP}0.
\]
On the event $|\lambda_k(B_N)|\le R$, the remaining series is uniformly
summable in the following sense:
\begin{equation}\label{eq:series-tail}
 \frac1N\sum_{p>L}\left|\frac{\Tr B_N^p-\sum_{\ell=1}^{k-1}\lambda_{\ell}^p}{pr^p}\right|
 \le\frac{2|E_N|-k+1}{N}\sum_{p>L}\frac{(R/r)^p}{p}.
\end{equation}
The deterministic series tail tends to zero
as $L\to\infty$. Taking $N\to\infty$ for a fixed $L$, then
$L\to\infty$, shows that the probability that
$|\lambda_k(B_N)|\le R$ and the absolute value of the series in
\eqref{eq:trace-log} exceeds any fixed positive tolerance tends to zero.

On $\Omega_N(r,\delta)$, for each $\ell\in [k-1]$,
\[
 \delta\le|r-\lambda_\ell|\le r+\Delta_N.
\]
As $\Delta_N=O_{\PP}(\log N)$,
$N^{-1}\log|r-\lambda_\ell|$ also tends to zero on this event, in precisely
the same probability-of-intersection sense.
Equation \eqref{eq:trace-log} and $2|E_N|/N\to m $ therefore imply
\begin{equation}\label{eq:Bdet-limit}
 \PP\left(\Omega_N(r,\delta)\cap
 \left\{\left|\frac1N\log|\det(rI-B_N)|- m \log r\right|
 >\varepsilon\right\}\right)\longrightarrow0.
\end{equation}
Finally, by \eqref{eq:ihara},
\[
 \frac1N\log|\det H_N(r)|
 =\frac1N\log|\det(rI-B_N)|
 -\left(\frac{|E_N|}{N}-1\right)\log(r^2-1).
\]
Combine this with that $|E_N|/N\to m /2$ in probability to obtain
\eqref{eq:localized}.
\end{proof}

\subsection{The limiting tree operator}
The majority of the remainder of the paper is dedicated towards showing that for $1<r<\sqrt{\kappa}$, $N^{-1}\log|\det H_N(r)|$ must be bounded away from $F_0(r)$.

Consider the tree $(T,o)$ defined above with degree distribution $D$.  If $p(v)$ is the parent of
$v\ne o$, in the tree, we can perform the expansion
\begin{equation}\label{eq:tree-form}
 \langle f,H(r,T)f\rangle
 =(r^2-1)|f(o)|^2+\sum_{v\ne o}|rf(v)-f(p(v))|^2.
\end{equation}
Therefore, for $r>1$, $H(r,T)$ is positive semidefinite. We then define $\nu_r$ to be the expected spectral measure on $H(r,T)$. Specifically, we set $\nu_r$ to be the probability measure supported on $[0,\infty)$ such that for all $\eta>0$,
\be\label{eq:nu}
 \int_{[0,\infty)}\frac{\nu_r(dx)}{x+\eta}
 =\bE\bigl\langle\delta_o,(H(r)+\eta I)^{-1}\delta_o\bigr\rangle.
\ee
The existence of such a measure exists by the spectral theorem, see \cite[Appendix A.2]{aizenman2015random}.

We introduce some properties of this measure.

\begin{lemma}\label{lem:root-bound}
For every $r>1$,
\begin{equation}\label{eq:root-moments}
 \int_{\bR} x\,\nu_r(dx)=(r^2-1)+ m ,
 \qquad
 \int_{\bR}\frac1x\,\nu_r(dx)\le\frac1{r^2-1},
 \qquad
 \int_{\bR}|\log x|\,\nu_r(dx)<\infty.
\end{equation}
In particular, $\nu_r(\{0\})=0$.
\end{lemma}

\begin{proof}
The first identity is the expected root diagonal entry.
For a positive operator $M\succ0$, we recall the diagonal variational formula. For any vector $f\in L^2(T)$, by Cauchy-Schwarz, 
\[
 |f(o)|^2
 =|\langle M^{-1/2}\delta_o,M^{1/2}f\rangle|^2
 \le\langle\delta_o,M^{-1}\delta_o\rangle\langle f,Mf\rangle,
\]
where equality holds for $f$ being any scalar times $M^{-1}\delta_o$.
Therefore,
\begin{equation}\label{eq:variational}
 \langle\delta_o,M^{-1}\delta_o\rangle
 =\sup_{f\ne0}\frac{|f(o)|^2}{\langle f,Mf\rangle}
 =\left(\inf_{f(o)=1}\langle f,Mf\rangle\right)^{-1},
\end{equation}

Equation \eqref{eq:tree-form} implies that for any $f$,
\[
 \langle f,(H(r)+\eta I)f\rangle\ge((r^2-1)+\eta)|f(o)|^2.
\]
Consequently,
\begin{equation}\label{eq:g-bound}
 \langle\delta_o,(H(r)+\eta I)^{-1}\delta_o\rangle
 \le\frac1{(r^2-1)+\eta}.
\end{equation}
Take expectations and let $\eta\downarrow0$ to prove the inverse-moment
bound. Finally, use $\log x\le x$ for $x\ge1$ and
$-\log x\le x^{-1}$ for $0<x\le1$ to prove the last bound in \eqref{eq:root-moments}.
\end{proof}

\subsection{Convergence of the Bethe-Hessian measure}

Let $\nu_{N,r}$ be the probability measure induced by the empirical spectral measure of $H_N(r)$. The following is  consequence of the proof of \cite[Theorem 2]{bordenave2010resolvent} by setting their parameter $\alpha$ to $\frac1r$, rather than the values $\alpha\in \{0,1\}$ considered in their work. For completeness, we produce a proof in \Cref{sec:otherproofs}.

\begin{lemma}\label{lem:esd}
For every fixed $r>1$,
\begin{equation}\label{eq:esd}
 \nu_{N,r}\xrightarrow{\PP}\nu_r\quad\text{weakly}.
\end{equation}
All fixed moments converge in probability. In particular,
\begin{equation}\label{eq:second-moment}
 \frac1N\Tr [H_N(r)^2]
 \xrightarrow{\PP}\bE [D^2]+(2(r^2-1)+r^2) m +(r^2-1)^2.
\end{equation}
\end{lemma}

This implies a one-sided bound on the log-determinant.

\begin{lemma}\label{lem:log-upper}
For any fixed $r>1$ and $\epsilon>0$,
\begin{equation}\label{eq:log-upper}
 \PP\left(\frac1N\log|\det H_N(r)|-\int\log x\,\nu_r(dx)>\epsilon\right)\to 0,
\end{equation}
using the convention that $\log0=-\infty$.
\end{lemma}

\begin{proof}
For $\delta>0$, set $f_\delta(x)=\log\max\{|x|,\delta\}$. Weak convergence
and the second-moment bound imply
\begin{equation}\label{eq:truncated-log}
 \int f_\delta\,d\nu_{N,r}\xrightarrow{\PP}\int f_\delta\,d\nu_r.
\end{equation}
Indeed, truncate this continuous function above a large constant and use
\[
 \sup_{|x|>L}\frac{\log|x|}{x^2}\longrightarrow0
 \quad(L\to\infty)
\]
to control its positive tail by the second spectral moment.
For every $N$,
\[
 \frac1N\log|\det H_N(r)|
 =\int\log|x|\,\nu_{N,r}(dx)
 \le\int f_\delta\,d\nu_{N,r}.
\]
\Cref{lem:root-bound} gives
$\int f_\delta\,d\nu_r\to\int\log x\,\nu_r(dx)$ as
$\delta\downarrow0$. Choose $\delta$ so that this integral is smaller
than $\int\log x\,\nu_r(dx)+\epsilon$, and apply \eqref{eq:truncated-log}.
\end{proof}

The chief technical result we are left to prove is a bound on the log-determinant of the limiting spectral measure.

\begin{prop}\label{prop:strict}
For every $1<r<\sqrt\kappa$,
\begin{equation}\label{eq:strict}
 \int_{\bR}\log x\,\nu_r(dx)
 < m \log r-\left(\frac m 2-1\right)\log(r^2-1).
\end{equation}
\end{prop}

Before proving \Cref{prop:strict}, we will give the proof of the main theorem assuming it.
\begin{proof}[Proof of \Cref{thm:main}]
It suffices to show that for our fixed $k\geq 1$,
\[
 \PP(|\lambda_k(B_N)|\le R)\longrightarrow0
 \quad\text{for every }1<R<\sqrt\kappa.
\]

Fix $R$ and choose $k$ different numbers $r_1,\ldots r_k$ and error parameter $\delta>0$ such that
\begin{equation}\label{eq:two-radii}
\sqrt\kappa>r_1>r_2>\cdots>r_k>R
 \qquad \delta\deq \frac12\min\left\{\{r_{\ell}-r_{\ell+1}\}_{\ell\in [k-1]},r_k-R\right\}>0.
\end{equation}
A set of $k-1$ numbers cannot be within distance strictly less than $\delta$ of all of $r_1,\ldots,r_k$. Therefore,
\begin{equation}\label{eq:cover}
 \{|\lambda_k(B_N)|\le R\}
 \subset \bigcup_{\ell=1}^{k} \Omega_N(r_\ell,\delta)
\end{equation}
apart from the already negligible event that $B_N$ has fewer than $k$ eigenvalues.

For any $r\in \{r_1,\ldots, r_{k}\}$,
\Cref{prop:strict} gives
\[
 \int_{\bR}\log x\,\nu_{r}(dx)<F_0(r),
\]
meaning we can set error parameter
\[
\epsilon\deq F_0(r)- \int_{\bR}\log x\,\nu_{r}(dx)>0.
\]
\Cref{lem:log-upper} implies
\begin{equation}\label{eq:upper-event}
 \PP\left(\frac1N\log|\det H_N(r)|> \int_{\bR}\log x\,\nu_{r}(dx)+\epsilon/2\right)\longrightarrow0.
\end{equation}
Thus \Cref{lem:forced}, applied with an error tolerance
$\epsilon/2$, implies
\begin{equation}\label{eq:lower-event}
 \PP\left(\Omega_N(r,\delta)\cap
 \left\{\frac1N\log|\det H_N(r)|\le \int_{\bR}\log x\,\nu_{r}(dx)+\epsilon\right\}\right)
 \longrightarrow0.
\end{equation}
The events in \eqref{eq:upper-event} and \eqref{eq:lower-event} cover
$\Omega_N(r,\delta)$, so its probability tends to zero for each $r_\ell$.
Equation \eqref{eq:cover} proves the desired bound.
\end{proof}
The rest of the paper is to prove \Cref{prop:strict}.

\section{The tree logarithmic determinant}
\subsection{Root and cavity variables}
The next goal is to estimate the log-determinant on the tree. To do this, we examine the resolvent on the tree and examine its recursive properties.
We consider the root $o$ with offspring $v$. Define $H_{T_v}^{(o)}$ to be the submatrix of $H_{T}$ on the forward tree rooted at $v$. Note this is the Bethe-Hessian on $T_v$, except it is one greater on the entry corresponding to $v$.

For $\eta>0$, define
\begin{equation}\label{eq:cavity-def}
 h_\eta\deq \langle\delta_v,(H_{T_v}^{(o)}+\eta I)^{-1}\delta_v\rangle,
 \qquad X_\eta\deq1-r^2h_\eta.
\end{equation}
Bounding these will allow us to understand the spectral measure. For $K$ from \eqref{eq:sizebiased-law}, the  descendant subtree has offspring at each vertex with distribution $K$. Taking the Schur complement gives
\begin{equation}\label{eq:cavity-eta}
 h_\eta=\frac1{r^2+\eta+\sum_{i=1}^{K}X_{\eta,i}},
 \qquad
 X_\eta=\frac{\eta+\sum_{i=1}^{K}X_{\eta,i}}
                 {r^2+\eta+\sum_{i=1}^{K}X_{\eta,i}}.
\end{equation}
The child variables are independent copies of $X_\eta$, independent of
$K$. These equalities hold on the descendant tree, or in distribution
when the tree coupling is suppressed. The variational bound \eqref{eq:tree-form} gives
\begin{equation}\label{eq:cavity-bounds}
 0<h_\eta\le\frac1{r^2+\eta},
 \qquad \frac{\eta}{r^2+\eta}\le X_\eta<1.
\end{equation}
At the original root the corresponding identity is instead
\begin{equation}\label{eq:root-g}
 g_\eta:=\langle\delta_o,(H(r)+\eta I)^{-1}\delta_o\rangle
 =\frac1{r^2-1+\eta+\sum_{i=1}^{D}X_{\eta,i}}.
\end{equation}
In particular, the root sum has $D$ terms, while the cavity sum has $K$ terms.

We specify the boundary condition used in \eqref{eq:cavity-eta}.
On the $K$-offspring tree start with $X_{v,0}(\eta)\deq1$, and set
\begin{equation}\label{eq:iteration}
 X_{v,\ell+1}(\eta)
\deq \frac{\eta+\sum_{w:\,p(w)=v}X_{w,\ell}(\eta)}
        {r^2+\eta+\sum_{w:\,p(w)=v}X_{w,\ell}(\eta)}.
\end{equation}

These variables decrease in $\ell$. The value $X=1$ outside the compression means that the resolvent of an omitted child subtree is 0.

As $\eta\downarrow0$, these variables decrease. The decreasing limits in
depth and in $\eta$ commute, so
\begin{equation}\label{eq:zero-cavity}
 X:=\lim_{\eta\downarrow0}X_\eta
 =\lim_{\ell\to\infty}X_{o,\ell}(0),
 \qquad
 X\stackrel d=\frac{\sum_{i=1}^{K}X_i}{r^2+\sum_{i=1}^{K}X_i}.
\end{equation}
This is the maximal solution on $[0,1]$, obtained by decreasing iteration
from $1$, and
\begin{equation}\label{eq:Xbounds}
 0\le X\le\frac K{r^2+K}<1\quad\text{almost surely}.
\end{equation}

\subsection{Size biasing and the logarithmic identity}
The main goal of this subsection is to prove a probabilistic interpretation of the log-determinant of the tree.

The fundamental equality (see, e.g. \cite{kurtz1997conceptual}) is that for integrable $f$, 
\begin{equation}\label{eq:sizebias}
 \bE\left[\sum_{i=1}^{D} f\left(X_i,\sum_{j\ne i}X_j\right)\right]
 = m \,\bE \left[f\left(X,\sum_{j=1}^{K}X_j\right)\right],
\end{equation}
where all variables on the right are independent. Next, we show that we can decompose the logarithmic moment of the spectral measure $\nu_r$ into  contributions from vertices and edges. 

\begin{prop}\label{prop:tree-log}
Let $X$ and $Y$ be independent copies of the maximal solulution of \eqref{eq:zero-cavity}. Then
\begin{equation}\label{eq:tree-log}
 \int_{\bR}\log x\,\nu_r(dx)
 =\bE\log\left(r^2-1+\sum_{i=1}^{D}X_i\right)
 -\frac m 2\bE\log\left(1-\frac{(1-X)(1-Y)}{r^2}\right).
\end{equation}
\end{prop}

\begin{proof}
For $\eta>0$, take independent cavity variables $X_\eta,Y_\eta$, coupled
with their derivatives on their respective trees, and set
\begin{equation}\label{eq:Feta}
 F(\eta)\deq \bE\log\left(r^2-1+\eta+\sum_{i=1}^{D}X_{\eta,i}\right)
 -\frac m 2\bE\log\left(1-\frac{(1-X_\eta)(1-Y_\eta)}{r^2}\right).
\end{equation}
The first logarithm is integrable because the internal sum is at most $D$.
The second logarithm is bounded, since its argument lies in $[1-1/r^2,1]$.

Recall $g_\eta$ from \eqref{eq:root-g}. We claim that\begin{equation}\label{eq:Fprime}
 F'(\eta)=\bE g_\eta.
\end{equation}
We write $\dot X_\eta\deq \frac{\partial}{\partial \eta}X_\eta$. Then
\begin{equation}\label{eq:Xderivative}
 0\le\dot X_\eta
 =r^2\langle\delta_v,(H_{T_v}^{(o)}+\eta I)^{-2}\delta_v\rangle
 \le r^2\eta^{-2}.
\end{equation}
Together with $\bE D<\infty$ and the positive lower bounds on the
logarithmic arguments, this justifies differentiating under expectations
on compact subintervals of $(0,\infty)$.
Next we apply \eqref{eq:sizebias} to the derivative of first term in \eqref{eq:Feta} and use symmetry in the edge term. We obtain
\begin{align}\label{eq:cancellation}
 F'(\eta)={}&\bE g_\eta
 + m \,\bE\frac{\dot X_\eta}
 {r^2-1+\eta+X_\eta+\sum_{i=1}^{K}X_{\eta,i}}\notag\\
 &- m \,\bE\frac{\dot X_\eta(1-Y_\eta)}
 {r^2-(1-X_\eta)(1-Y_\eta)}.
\end{align}
In the middle expectation the $K$-sum is independent of
$(X_\eta,\dot X_\eta)$. Couple it with
\[
 Y_\eta=\frac{\eta+\sum_{i=1}^{K}X_{\eta,i}}
                 {r^2+\eta+\sum_{i=1}^{K}X_{\eta,i}}.
\]
Then
\[
 r^2-1+\eta+X_\eta+\sum_{i=1}^{K}X_{\eta,i}
 =\frac{r^2-(1-X_\eta)(1-Y_\eta)}{1-Y_\eta}.
\]
The last two terms of \eqref{eq:cancellation} cancel, proving \eqref{eq:Fprime}.

In comparison,
\[
 \frac{d}{d\eta}\int_{\bR}\log(x+\eta)\,\nu_r(dx)
 =\int_{\bR}\frac1{x+\eta}\,\nu_r(dx)=\bE g_\eta.
\]
As we always have $0\leq g_\eta\leq \frac{1}{r^2-1}$, this means that for some constant $C$ and $\eta\in (0,\infty)$,
\[
\int_\bR \log(x+\eta)\nu_r(\rd x)=F(\eta)+C.
\]
Both equal
$\log\eta+o_\eta(1)$ as $\eta\to\infty$, so $C=0$.

Finally, let $\eta\downarrow0$. \Cref{lem:root-bound} supplies the
finite absolute logarithmic moment on the left of \eqref{eq:tree-log}. On the right, the first logarithm in \eqref{eq:Feta} is bounded between $\log (r^2-1)$ and $\log(r^2-1+1+D)$ for $\eta\le1$,
and the second logarithm is uniformly bounded. Equation
\eqref{eq:zero-cavity} and dominated convergence give \eqref{eq:tree-log}.
\end{proof}

\subsection{A finite-energy flow on the forward tree}
In order to control the spectral measure, our next goal is to show that the variable $X$ is nontrivial.
\begin{lemma}\label{lem:nonzero}
    If $1<r<\sqrt{\kappa}$, the maximal solution of \eqref{eq:zero-cavity} satisfies
\begin{equation}\label{eq:nonzero}
 \PP(X>0)>0.
\end{equation}
\end{lemma}

\begin{proof}
Throughout this proof, we consider the cavity tree with root $o$, which is a random tree where each vertex has offspring law $K$.  Let $\sigma^2=\Var K<\infty$, which follows from the exponential
moment of $D$. For a vertex $v$, let $Z_{v,\ell}$ be its number of
descendants at distance $\ell$. The normalized generation sizes
$\kappa^{-\ell}Z_{v,\ell}$ are nonnegative martingales. By \cite
[Theorem~8.1]{harris1963theory}, the martingales are bounded in $L^2$, and their limits $W_v$ satisfy
\begin{equation}\label{eq:W}
 \bE W_o=1,\qquad
 \bE W_o^2=1+\frac{\sigma^2}{\kappa(\kappa-1)},\qquad
 W_v=\frac1\kappa\sum_{w:\,p(w)=v}W_w.
\end{equation}
The branching identity holds simultaneously at all vertices almost surely.
In particular, $\PP(W_o>0)>0$.

We now want to say that on the event that $W_o>0$, with nonzero probability $X>0$ as well. To do this, we will use the often useful electrical network interpretation of the graph. Through the proof, we will use a set of standard facts about these networks. For an overview, see \cite[Chapters II and IX]{bollobas1998modern}.

We will interpret the cavity iteration as an effective conductance, so that
we can prove its positivity by constructing a flow of finite energy.  For a vertex $v$, let $C_{v,\ell}$ be the effective conductance from $v$
to the set of its descendants at distance $\ell$, with those descendants wired
together. In defining this network, assign resistance $r^{2(p-1)}$ to
each edge entering depth $p$ {relative to $v$}. A branch that ends
before reaching the terminal generation carries no current and
contributes zero conductance. Set $C_{v,0}=\infty$, corresponding to
zero resistance between a vertex and itself.

Consider the branch from $v$ through a child $w$. The edge $(v,w)$ has
resistance $1$. Beyond this edge, every resistance is $r^2$ times the
resistance used in the definition of $C_{w,\ell-1}$. Thus the resistance
of the network beyond $w$ is $r^2/C_{w,\ell-1}$. Adding resistances in
series, the conductance of the entire branch through $w$ is
\[
 \left(1+\frac{r^2}{C_{w,\ell-1}}\right)^{-1}
 =\frac{C_{w,\ell-1}}{r^2+C_{w,\ell-1}},
\]
with value $0$ when $C_{w,\ell-1}=0$ and value $1$ when
$C_{w,\ell-1}=\infty$.
The branches through different children are in parallel, so their
conductances add. Comparing this recursion with the cavity iteration \eqref{eq:iteration},
and using the common initial value $X_{v,0}(0)=1$, gives by induction
\begin{equation}\label{eq:conductance}
 \begin{aligned}
 C_{v,\ell}
 &=\sum_{w:\,p(w)=v}
   \frac{C_{w,\ell-1}}{r^2+C_{w,\ell-1}}
   =\sum_{w:\,p(w)=v}X_{w,\ell-1}(0),\\
 X_{v,\ell}(0)&=\frac{C_{v,\ell}}{r^2+C_{v,\ell}}.
 \end{aligned}
\end{equation}
Write $C_\ell=C_{o,\ell}$. It remains to show that, with positive probability, $C_\ell$ stays
bounded away from zero as $\ell\to\infty$.

Set $|v|$ to be the distance from the root $o$ to $v$. On $\{W_o>0\}$, assign current
\[
 \theta(v)=\frac{\kappa^{-|v|}W_v}{W_o}
\]
to the edge entering $v$. The branching identity in
\eqref{eq:W} gives
\[
 \sum_{w:\,p(w)=v}\theta(w)
 =\frac{\kappa^{-|v|-1}}{W_o}
   \sum_{w:\,p(w)=v}W_w
 =\frac{\kappa^{-|v|}W_v}{W_o}=\theta(v).
\]
Thus, one unit of current leaves the root, and current is conserved at
every other vertex. This defines a unit flow to infinity.

By \cite[Equation IX.5]{bollobas1998modern}, the total energy of this flow is
\begin{equation}\label{eq:energy}
 \mathcal E
 :=\sum_{v\ne o}r^{2(|v|-1)}\theta(v)^2
 =\frac1{W_o^2}
   \sum_{p\ge1}r^{2(p-1)}\kappa^{-2p}
                   \sum_{|v|=p}W_v^2.
\end{equation}
We want to prove that $\mathcal E$ is finite on $\{W_o>0\}$.
Conditional on the tree exposed through depth $p$, the descendant
trees rooted at that depth are independent copies of the original
forward tree. Hence
\[
 \bE\sum_{|v|=p}W_v^2
 =\kappa^p\bE W_o^2.
\]
Using Tonelli's theorem and $r^2<\kappa$, we obtain
\begin{align}\label{eq:expected-energy}
 \bE\sum_{p\ge1}r^{2(p-1)}\kappa^{-2p}
                       \sum_{|v|=p}W_v^2
 &=\bE W_o^2\sum_{p\ge1}r^{2(p-1)}\kappa^{-p}\notag\\
 &=\frac{\bE W_o^2}{\kappa}
   \sum_{j\ge0}\left(\frac{r^2}{\kappa}\right)^j
 =\frac{\bE W_o^2}{\kappa-r^2}<\infty.
\end{align}
Therefore, on $\{W_o>0\}$,
$\mathcal E<\infty$ almost surely.

On $\{W_o>0\}$, restrict the unit flow $\theta$ to the first
$\ell$ generations. This gives a unit flow from $o$ to the wired
$\ell$th generation. By Thomson's principle
\cite[Theorem IX.2]{bollobas1998modern}, the effective resistance
is at most the energy of this flow:
\[
 \frac1{C_\ell}
 \le \sum_{1\le |v|\le\ell}
       r^{2(|v|-1)}\theta(v)^2
 \le \mathcal E.
\]
Consequently, \eqref{eq:conductance} gives
\[
 X_{o,\ell}(0)
 =\frac{C_\ell}{r^2+C_\ell}
 \ge \frac1{1+r^2\mathcal E}
 \qquad\text{for every }\ell\ge1.
\]
Since $\mathcal E<\infty$ almost surely on $\{W_o>0\}$,
letting $\ell\to\infty$ in \eqref{eq:zero-cavity} yields
\[
 X\ge\frac1{1+r^2\mathcal E}>0
 \qquad\text{almost surely on }\{W_o>0\}.
\]
As $\PP(W_o>0)>0$, this proves \eqref{eq:nonzero}.
\end{proof}

\subsection{The interpolation}

At this point, we are ready to prove our explicit bound. We do so through an interpolation argument.
\begin{proof}[Proof of \Cref{prop:strict}]
Let $X$ and $Y$ be independent copies sampled according to the law in \eqref{eq:zero-cavity}.
For $0\le t\le1$, define
\begin{equation}\label{eq:Phi}
 \Phi(t)=\bE\log\left(r^2-1+t\sum_{i=1}^{D}X_i\right)
 -\frac m 2\bE\log\left(\frac{r^2-(1-tX)(1-tY)}{r^2}\right).
\end{equation}
Its endpoints are
\begin{align}
 \Phi(0)&=\frac m 2\log r^2+\left(1-\frac m 2\right)\log (r^2-1)
 = m \log r-\left(\frac m 2-1\right)\log(r^2-1),
 \label{eq:Phi0}\\
 \Phi(1)&=\int_{\bR}\log x\,\nu_r(dx),
 \label{eq:Phi1}
\end{align}
where \eqref{eq:Phi1} is \Cref{prop:tree-log}.

We compute the sign of the derivative $\Phi'(t)$. The derivative of the first term in \eqref{eq:Phi} is bounded
by $D/(r^2-1)$, an integrable variable. Also
\[
 r^2-1\le r^2-(1-tX)(1-tY)\le r^2+1,
\]
so the edge derivative is uniformly bounded. These estimates justify
differentiation under expectations and continuity on $[0,1]$.
\eqref{eq:sizebias} gives
\[
 \bE\frac{\sum_{i=1}^{D}X_i}{r^2-1+t\sum_{i=1}^{D}X_i}
 = m \,\bE\frac{X}{r^2-1+tX+t\sum_{i=1}^{K}X_i}.
\]
On the right the $K$-sum is independent of $X$. By the cavity recursion,
it has the same law as $r^2Y/(1-Y)$. Symmetry in $X,Y$ combines the two edge
derivatives, giving
\begin{equation}\label{eq:Phi-prime-first}
 \frac{\Phi'(t)} m 
 =\bE\left[
 \frac{X}{r^2-1+tX+tr^2Y/(1-Y)}
 -\frac{X(1-tY)}{r^2-1+t(X+Y)-t^2XY}
 \right].
\end{equation}
We can write
\[
 \frac{r^2-1+t(X+Y)-t^2XY}{1-tY}
 =r^2-1+tX+\frac{tr^2Y}{1-tY}.
\]
Consequently,
\begin{equation}\label{eq:Phi-prime}
 \frac{\Phi'(t)} m 
 =\bE\left[
 \frac{X}{r^2-1+tX+tr^2Y/(1-Y)}
 -\frac{X}{r^2-1+tX+tr^2Y/(1-tY)}
 \right].
\end{equation}
For $0<t<1$, the first denominator is at least the second, and is strictly
larger if $Y>0$. The integrand is therefore nonpositive and is strictly
negative on $\{X>0,Y>0\}$. Since $r^2<\kappa$, \Cref{lem:nonzero}
and independence give this event positive probability. Thus
$\Phi'(t)<0$ throughout $(0,1)$, proving $\Phi(1)<\Phi(0)$ and hence
\eqref{eq:strict}.
\end{proof}

\appendix
\section{Deferred proof}\label{sec:otherproofs}
\begin{proof}[Proof of \Cref{lem:esd}]
For any finite graph put $H=D_G-rA_G+(r^2-1)I$ and
$M=(r+1)D_G+(r^2-1)I$. The inequality
$|\langle f,A_Gf\rangle|\le\langle f,D_Gf\rangle$ gives
$-M\preceq H\preceq M$. Therefore, we can write $H=M^{1/2}QM^{1/2}$ for a self-adjoint
contraction $-I\preceq Q\preceq I$. Schatten H\"older gives
\begin{equation}\label{eq:schatten}
 \Tr[|H|^{2p}]\le\Tr [M^{2p}],\qquad p\ge1.
\end{equation}
Thus
\begin{equation}\label{eq:matrix-moment-bound}
 \frac1N\Tr|H_N(r)|^{2p}
 \le\frac1N\sum_v\bigl((r+1)\deg(v)+r^2-1\bigr)^{2p}
 \xrightarrow{\PP}\bE\bigl[(r+1)D+r^2-1\bigr]^{2p},
\end{equation}
where the convergence applies to the expression on the right of the
inequality, by \eqref{eq:local}.

For fixed $p\ge1$, the quantity $(H_N(r)^p)_{vv}$ is a $p$-local
function. We first check its second moments on the limiting tree.
Apply \eqref{eq:local} to bounded truncations of the nonnegative
local function $(H_N(r)^{2p})_{vv}$. By
\eqref{eq:matrix-moment-bound}, their empirical averages are bounded
above by a quantity converging in probability to
$\bE\left[((r+1)D+r^2-1)^{2p}\right]$. Passing to the limit and then removing
the truncation by monotone convergence gives
\begin{equation}\label{eq:tree-moment-bound}
 \bE\left[\langle\delta_o,H^{2p}\delta_o\rangle\right]
 \le \bE\bigl[((r+1)D+r^2-1)^{2p}\bigr]<\infty.
\end{equation}

Cauchy-Schwarz and \eqref{eq:matrix-moment-bound} give
\[
 \frac1N\sum_v |(H_N(r)^p)_{vv}|^2
 \le \frac1N\Tr \left[H_N(r)^{2p}\right]=O_{\PP}(1).
\]
Similarly, \eqref{eq:tree-moment-bound} gives
\[
 \bE\left[\bigl|\langle\delta_o,H^p\delta_o\rangle\bigr|^2\right]
 \le \bE\left[\langle\delta_o,H^{2p}\delta_o\rangle\right]<\infty.
\]
Thus the local function $(H_N(r)^p)_{vv}$ satisfies the
second-moment conditions in \eqref{eq:local}. Applying that
assumption directly yields
\begin{equation}\label{eq:moment-convergence}
 \frac1N\Tr \left[H_N(r)^p\right]
 \xrightarrow{\PP}
 \bE\left[\langle\delta_o,H^p\delta_o\rangle\right]
 =\int x^p\,\nu_r(dx)
 \qquad\text{for every fixed }p\ge1.
\end{equation}

To check moment determinacy, set $b=a/(r+1)>0$. The elementary bound
$y^{2p}\le(2p)!b^{-2p}e^{by}$ for $y\ge0$ and
\eqref{eq:local} give
\begin{equation}\label{eq:carleman-bound}
 \int x^{2p}\,\nu_r(dx)
 \le(2p)!b^{-2p}e^{b(r^2-1)}\bE [e^{aD}].
\end{equation}
As the exponential moment is bounded, Carleman's condition follows. From every subsequence, we can extract a
further subsequence on which all the moments in
\eqref{eq:moment-convergence} converge almost surely. The even moments
give tightness and allow moments to pass to every weak limit. Moment
determinacy identifies that limit with $\nu_r$, proving \eqref{eq:esd}.

Finally,
\[
 \frac1N\Tr \left[H_N(r)^2\right]
 =\frac1N\sum_v(\deg(v)+r^2-1)^2+\frac {r^2}N\sum_v\deg(v),
\]
and \eqref{eq:local} proves \eqref{eq:second-moment}.
\end{proof}
\bibliographystyle{amsplain}
\bibliography{ref}

@book{harris1963theory,
  title={The theory of branching processes},
  author={Harris, Theodore Edward},
  volume={6},
  year={1963},
  publisher={Springer Berlin}
}

@incollection {kurtz1997conceptual,
    AUTHOR = {Kurtz, Thomas and Lyons, Russell and Pemantle, Robin and
              Peres, Yuval},
     TITLE = {A conceptual proof of the {K}esten-{S}tigum theorem for
              multi-type branching processes},
 BOOKTITLE = {Classical and modern branching processes ({M}inneapolis, {MN},
              1994)},
    SERIES = {IMA Vol. Math. Appl.},
    VOLUME = {84},
     PAGES = {181--185},
 PUBLISHER = {Springer, New York},
      YEAR = {1997},
      ISBN = {0-387-94872-4},
   MRCLASS = {60J80},
  MRNUMBER = {1601737},
       DOI = {10.1007/978-1-4612-1862-3\_14},
       URL = {https://doi.org/10.1007/978-1-4612-1862-3_14},
}

@book{bollobas1998modern,
  title={Modern graph theory},
  author={Bollob{\'a}s, B{\'e}la},
  volume={184},
  year={1998},
  publisher={Springer Science \& Business Media}
}

@article{alon1986eigenvalues,
  title = {Eigenvalues and expanders},
  author = {Alon, Noga},
  journal = {Combinatorica},
  volume = {6},
  number = {2},
  pages = {83--96},
  year = {1986},
  doi = {10.1007/BF02579166}
}

@article{nilli1991second,
  title = {On the second eigenvalue of a graph},
  author = {Nilli, A.},
  journal = {Discrete Mathematics},
  volume = {91},
  number = {2},
  pages = {207--210},
  year = {1991},
  doi = {10.1016/0012-365X(91)90112-F}
}

@article{hoory2005lower,
  title = {A lower bound on the spectral radius of the universal cover of a graph},
  author = {Hoory, Shlomo},
  journal = {Journal of Combinatorial Theory, Series B},
  volume = {93},
  number = {1},
  pages = {33--43},
  year = {2005},
  doi = {10.1016/j.jctb.2004.06.001}
}

@article{jiang2019spectral,
  title = {On spectral radii of unraveled balls},
  author = {Jiang, Zilin},
  journal = {Journal of Combinatorial Theory, Series B},
  volume = {136},
  pages = {72--80},
  year = {2019},
  doi = {10.1016/j.jctb.2018.09.003}
}

@incollection{hashimoto1989zeta,
  title = {Zeta functions of finite graphs and representations of {$p$}-adic groups},
  author = {Hashimoto, Ki-ichiro},
  booktitle = {Automorphic Forms and Geometry of Arithmetic Varieties},
  series = {Advanced Studies in Pure Mathematics},
  volume = {15},
  pages = {211--280},
  publisher = {Academic Press},
  year = {1989},
  doi = {10.2969/aspm/01510211}
}

@article{bass1992ihara,
  title = {The {Ihara--Selberg} zeta function of a tree lattice},
  author = {Bass, Hyman},
  journal = {International Journal of Mathematics},
  volume = {3},
  number = {6},
  pages = {717--797},
  year = {1992},
  doi = {10.1142/S0129167X92000357}
}

@article{banks2024useful,
  title={A useful formula for periodic Jacobi matrices on trees},
  author={Banks, Jess and Breuer, Jonathan and Garza-Vargas, Jorge and Seelig, Eyal and Simon, Barry},
  journal={Proceedings of the National Academy of Sciences},
  volume={121},
  number={23},
  pages={e2315218121},
  year={2024},
  publisher={National Academy of Sciences}
}

@article{anantharaman2017some,
  title={Some relations between the spectra of simple and non-backtracking random walks},
  author={Anantharaman, Nalini},
  journal={arXiv preprint arXiv:1703.03852},
  year={2017}
}

@article{sodin2007random,
  title = {Random matrices, nonbacktracking walks, and orthogonal polynomials},
  author = {Sodin, Sasha},
  journal = {Journal of Mathematical Physics},
  volume = {48},
  number = {12},
  pages = {123503},
  year = {2007},
  doi = {10.1063/1.2819599}
}

@article{krzakala2013spectral,
  title = {Spectral redemption in clustering sparse networks},
  author = {Krzakala, Florent and Moore, Cristopher and Mossel, Elchanan and Neeman, Joe and Sly, Allan and Zdeborov{\'a}, Lenka and Zhang, Pan},
  journal = {Proceedings of the National Academy of Sciences},
  volume = {110},
  number = {52},
  pages = {20935--20940},
  year = {2013},
  doi = {10.1073/pnas.1312486110}
}

@inproceedings{saade2014spectral,
  title = {Spectral clustering of graphs with the {Bethe Hessian}},
  author = {Saade, Alaa and Krzakala, Florent and Zdeborov{\'a}, Lenka},
  booktitle = {Advances in Neural Information Processing Systems},
  volume = {27},
  pages = {406--414},
  year = {2014},
  url = {https://arxiv.org/abs/1406.1880}
}

@article{angel2015non,
  title = {The non-backtracking spectrum of the universal cover of a graph},
  author = {Angel, Omer and Friedman, Joel and Hoory, Shlomo},
  journal = {Transactions of the American Mathematical Society},
  volume = {367},
  number = {6},
  pages = {4287--4318},
  year = {2015},
  doi = {10.1090/S0002-9947-2014-06255-7}
}

@article{bordenave2020new,
  title = {A new proof of {Friedman}'s second eigenvalue theorem and its extension to random lifts},
  author = {Bordenave, Charles},
  journal = {Annales scientifiques de l'{\'E}cole normale sup{\'e}rieure},
  series = {4},
  volume = {53},
  number = {6},
  pages = {1393--1439},
  year = {2020},
  doi = {10.24033/asens.2450}
}

@article{lubetzky2016cutoff,
  title={Cutoff on all Ramanujan graphs},
  author={Lubetzky, Eyal and Peres, Yuval},
  journal={Geometric and Functional Analysis},
  volume={26},
  number={4},
  pages={1190--1216},
  year={2016},
  publisher={Springer}
}

@article {alon2007non,
    AUTHOR = {Alon, Noga and Benjamini, Itai and Lubetzky, Eyal and Sodin,
              Sasha},
     TITLE = {Non-backtracking random walks mix faster},
   JOURNAL = {Commun. Contemp. Math.},
  FJOURNAL = {Communications in Contemporary Mathematics},
    VOLUME = {9},
      YEAR = {2007},
    NUMBER = {4},
     PAGES = {585--603},
      ISSN = {0219-1997,1793-6683},
   MRCLASS = {60C05 (05C38 60G50)},
  MRNUMBER = {2348845},
MRREVIEWER = {Stanislav\ Volkov},
       DOI = {10.1142/S0219199707002551},
       URL = {https://doi.org/10.1142/S0219199707002551},
}

@article{Keller2017criticality,
  author  = {Keller, Matthias and Pinchover, Yehuda and Pogorzelski, Felix},
  title   = {Criticality theory for {Schr{\"o}dinger} operators on graphs},
  journal = {Journal of Spectral Theory},
  volume  = {10},
  number  = {1},
  pages   = {73--114},
  year    = {2020},
  doi     = {10.4171/JST/286}
}

@article{Lyons2010identities,
  author  = {Lyons, Russell},
  title   = {Identities and inequalities for tree entropy},
  journal = {Combinatorics, Probability and Computing},
  volume  = {19},
  number  = {2},
  pages   = {303--313},
  year    = {2010},
  doi     = {10.1017/S0963548309990605}
}

@article{aizenman2015random,
  title={Random operators},
  author={Aizenman, Michael and Warzel, Simone},
  journal={Graduate Studies in Mathematics},
  volume={168},
  pages={xiv+--326},
  year={2015},
  publisher={American Mathematical Society Providence, RI}
}

@article {bordenave2018nonbacktracking,
    AUTHOR = {Bordenave, Charles and Lelarge, Marc and Massouli\'e, Laurent},
     TITLE = {Nonbacktracking spectrum of random graphs: community detection
              and nonregular {R}amanujan graphs},
   JOURNAL = {Ann. Probab.},
  FJOURNAL = {The Annals of Probability},
    VOLUME = {46},
      YEAR = {2018},
    NUMBER = {1},
     PAGES = {1--71},
      ISSN = {0091-1798,2168-894X},
   MRCLASS = {05C80 (60B20 60J85 62M15)},
  MRNUMBER = {3758726},
MRREVIEWER = {D.\ Yogeshwaran},
       DOI = {10.1214/16-AOP1142},
       URL = {https://doi.org/10.1214/16-AOP1142},
}

@article{bordenave2010resolvent,
  title={Resolvent of large random graphs},
  author={Bordenave, Charles and Lelarge, Marc},
  journal={Random Structures \& Algorithms},
  volume={37},
  number={3},
  pages={332--352},
  year={2010},
  publisher={Wiley Online Library}
}

\end{document}